\documentclass[12pt,a4paper,oneside]{amsart}
\usepackage[left=1.1in, right=0.8in, top=1.0in, bottom=1.0in]{geometry}
\usepackage{graphicx}
\usepackage{setspace}
\usepackage{indentfirst} 
\usepackage{cases}
\usepackage{wrapfig}
\usepackage[toc,page]{appendix}
\usepackage{amsmath}
\usepackage{amsthm}
\usepackage{amssymb}
\usepackage{enumerate}
\usepackage{mathrsfs}
\usepackage{dcolumn}
\usepackage{hyperref}
\newcolumntype{L}{D{.}{.}{2,5}}
\theoremstyle{plain}
\newtheorem{thm}{Theorem}

\newtheorem{proposition}{Proposition}
\newtheorem{lemma}{Lemma}

\newtheorem{mydef}{Definition}

\DeclareMathOperator{\cl}{cl}
\DeclareMathOperator{\dom}{dom}

\DeclareMathOperator{\CAT}{CAT}
\newtheorem{remark}{Remark}
\theoremstyle{definition}

\begin{document}

\title[On a theorem about Mosco convergence]{\textbf On a theorem about \\Mosco convergence in Hadamard spaces}
\author{\textbf{Arian B\"erd\"ellima}}
% \date{May, 2020}
\thanks{This work was supported by Deutscher Akademischer Austauschdients (DAAD) and it is original results from author's Thesis.
Electronic address: berdellima@gmail.com}
\address{Institute for Numerical and Applied Mathematics}
\address{University of G\"ottingen}\address{37083 G\"ottingen, Germany}
 %\thanks{Electronic address: \texttt{berdellima@gmail.com}, MSC: 45A05, 31A05, 31A10, 31B05, 31B10 }
\maketitle

\begin{abstract}
Let  $(f^n),f$ be a sequence of proper closed convex functions defined on a Hadamard space. We show that the convergence of proximal mappings $J^n_{\lambda}x$ to $J_{\lambda}x$, under certain additional conditions, imply Mosco convergence of $f^n$ to $f$. This result is a converse to a theorem of Ba\v cak about Mosco convergence in Hadamard spaces.
  
\end{abstract} 

%\noindent \textbf{Keywords:} Hadamard space, Mosco convergence, {\em Moreau approximate}, proximal mapping, convex function, closed function

%%%%%%%%%%%%%%%%%%%%%%%%%%%%%%%%%%%%%%

%=====================================================
\section{Mosco Convergence in Hadamard Spaces}
\subsection{Hadamard spaces} A metric space $(X,d)$ is a $\CAT(0)$ space if it is geodesically connected,
and if every geodesic triangle $\Delta$ with vertices $p,q,r\in X$ and $x\in[p,r],y\in[p,q]$ we have $d(x,y)\leqslant\|\overline{x}-\overline{y}\|$, where $\overline{x}$ and $\overline{y}$ are the comparison points of $x$ and $y$ respectively in the comparison triangle $\overline{\Delta}$. 
Intuitively this means that $\Delta$ is at least as {\em thin} as its comparison triangle $\overline{\Delta}$ in the Euclidean plane (see Figure \ref{triangle}). A complete $\CAT(0)$ space is known as a Hadamard space. The importance of $\CAT(0)$ spaces
was recognized by Alexandrov \cite{Alexandrov} in the 1950s and that is why $\CAT(0)$ spaces are sometimes referred to as spaces of nonpositive curvature in the sense of Alexandrov (see Ballman \cite{Ballman}). The acronym $\CAT(0)$ was originally conceived by Gromov \cite{Gromov} where $C$ stands for Cartan, $A$ for Alexandrov and $T$ for Toponogov, and where $0$ is the upper
curvature bound.  For an extensive treatment of these spaces and the important role they play in mathematics one could refer to Bridson and Haefliger \cite{Brid} or D. Burago et al. \cite{Burago}.
\begin{figure}[h]
\label{triangle}
\caption{Geodesic triangle (left) and its comparison triangle (right)}
\includegraphics[width=10cm]{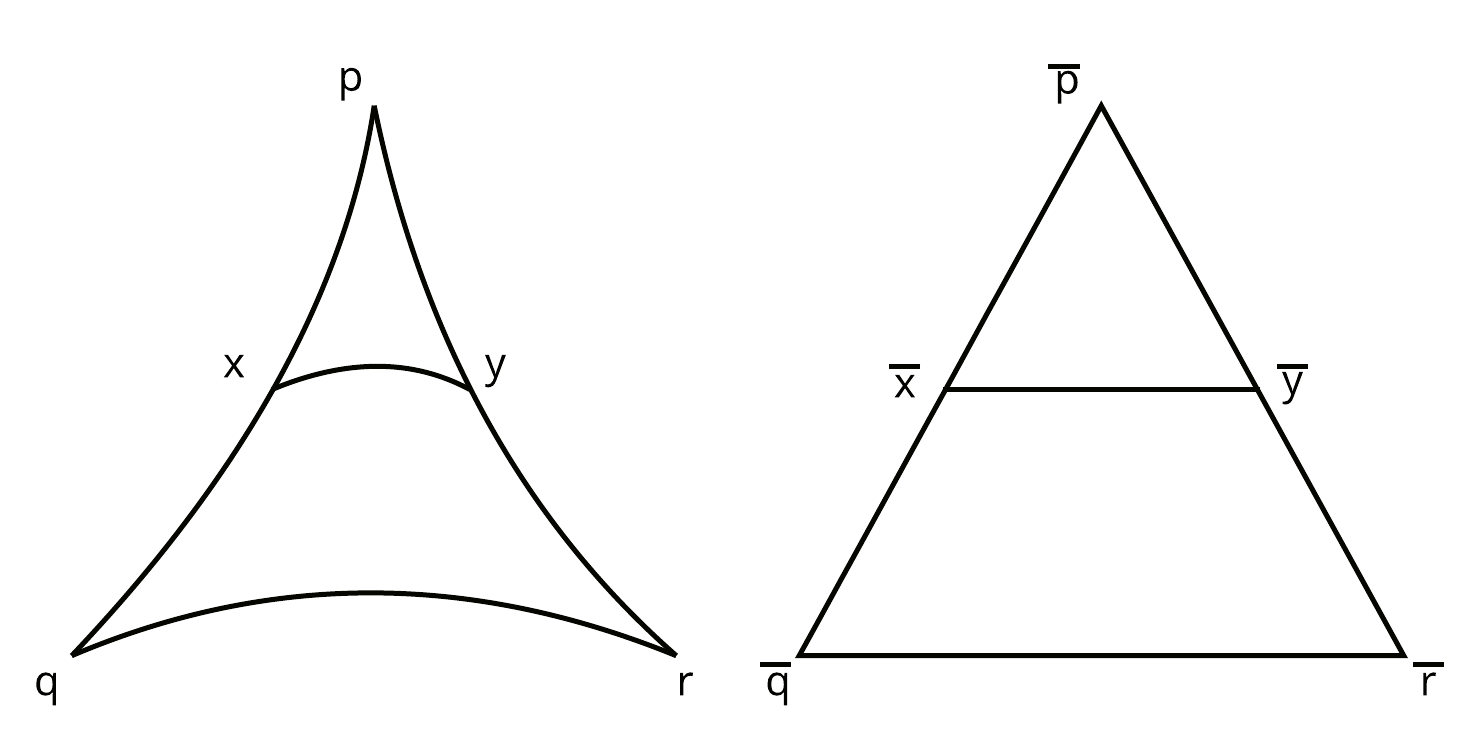}
\centering
\end{figure}

\subsection{Mosco convergence}
% \subsubsection{Mosco convergence of functions}
Let $(H,d)$ be a Hadamard space. A sequence of functions $f^n:H\to(-\infty,+\infty]$ is said to be Mosco convergent to $f: H\to(-\infty,+\infty]$ and we write $M-\lim_nf^n=f$ if for each $x\in H$:
\begin{enumerate}[(i)]
\item $f(x)\leqslant \liminf_nf^n(x_n)$ whenever $x_n\overset{w}\to x$\label{i}
\item there exists some sequence $(y_n)\subset H$ such that $y_n\to x$ and $f(x)\geqslant \limsup_nf^n(y_n)$.\label{ii}
\end{enumerate}
Note that $x_n\overset{w}\to x$ means $x_n$ converges {\em weakly} to $x$ and by definition $x_n\overset{w}\to x$ if and only if $\lim_nd(x,P_{\gamma}x_n)=0$ for every geodesic segment emanating from $x$. Here $P_{\gamma}x_n$ is the metric projection of $x_n$ onto $\gamma$. Weak limits are unique and every bounded sequence has a weakly convergent subsequence (see \cite[Chapter 2]{Berdellima}).
If \eqref{i} is substituted with strong convergence then one gets what is known as $\Gamma$-convergence. Therefore Mosco convergence is a stronger type of convergence and subsequently Mosco convergence implies  $\Gamma$-convergence. The original motivation for introducing Mosco convergence in analysis was to define a special convergence for closed convex sets of a normed space $X$, in which both the strong and the weak topologies of $X$ are involved (see \cite[Definition 1]{Mosco}). 
Another way to introduce Mosco convergence has been to make the so called \textit{Fenchel conjugate} $f^*$ of a \textit{closed convex proper} function $f$ bicontinous (see \cite[pg. 294]{Attouch}). Note that we say a function $f$ is a closed function whenever it is lower semicontinuous. Mosco convergence can be extended also for sets.
Let $\iota_S$ denote the indicator function of a set $S\subseteq H$ i.e. $\iota_S(x)=0$ if $x\in S$ and $\iota_S(x)=+\infty$ otherwise. A sequence of sets $(S_n)_{n\in\mathbb{N}}$  is said to converge in the sense of Mosco to a set $S$ whenever $M-\lim_n\iota_{S_n}=\iota_S$.

\begin{proposition}\label{prop}\cite[Corollary 5.2.8]{Bacak}
 Let $(H,d)$ be a Hadamard space and $(C_n)_{n\in\mathbb{N}}$ a sequence of closed convex sets. If $M-\lim_nC_n=C$ for some set $C\subseteq H$ then $C$ is closed and convex.
\end{proposition}
\begin{proof}
 By definition $M-\lim_nC_n=C$ means $M-\lim_n\iota_{C_n}=\iota_{C}$. $C_n$ is convex and closed for all $n$ implies that the indicator function $\iota_{C_n}$ is closed convex for all $n$. But Mosco convergence preserves convexity and lower semicontinuity therefore $\iota_C$ is a closed convex function. This is equivalent to $C$ being a closed convex set.
\end{proof}

\begin{proposition}\cite[Example 5.2.7]{Bacak}
\label{prop2}
 Let $(C_n)_{n\in\mathbb{N}}$ be a sequence of closed convex subsets of $H$. If $C_n$ is a nonincreasing sequence then $(C_n)$ Mosco converges to its intersection. If $(C_n)$ is nondecreasing then it Mosco converges to the closure of its union. 
\end{proposition}
\begin{proof}
 The proof follows the lines in \cite[Lemma 1.2, Lemma 1.3]{Mosco}. Let $C_n$ be a nonincreasing sequence of closed convex sets and $C:=\bigcap_kC_k$. 
 By definition it is suffices to prove that $M-\lim_n\iota_{C_n}=\iota_C$. 
 Let $(x_n)_{n\in\mathbb{N}}$ be a sequence such that $x_n\in C_n$ for all $n$ and $\lim_nx_n=x$. Then $\iota_{C_n}(x_n)=0$ for all $n$ implies $\limsup_n\iota_{C_n}(x_n)=0\leqslant\iota_C(x)$ confirming condition \eqref{ii}. Now let $(x_n)_{n\in\mathbb{N}}$ be such that $x_n\in C_n$ for all $w-\lim_nx_n=x$.  
 Assumption $C_n\subseteq C_m$ whenever $m\leqslant n$ implies that $x_n\in C_m$ for all $n\geqslant m$. But $C_m$ is a closed convex set hence by \cite[Lemma 3.1]{Bacak4} it follows that $C_m$ is weakly closed. Therefore $x=w-\lim_nx_n\in C_m$ and this holds for any $m\in\mathbb{N}$ since $m$ was arbitrary. This means that $x\in\bigcap_kC_k$ implying $\iota_C(x)=0\leqslant\liminf_n\iota_{C_n}(x_n)$ confirming condition \eqref{i}. Analogue arguments for the second statement. 
\end{proof}

\subsection{A theorem of Attouch}
Let $X$ be a \textit{normed linear} space and $f:X\to(-\infty,+\infty]$ a proper closed convex function. For $\lambda>0$ the \textit{Moreau-Yosida approximate} of $f$ is defined as 
\begin{equation}
 \label{eq:MYosida}
 f_{\lambda}(x):=\inf_{y\in X}\Big\{f(y)+\frac{1}{2\lambda}\|x-y\|^2\Big\}.
\end{equation}
It can be shown that $f_{\lambda}$ is a convex continuous function \cite[Proposition 3.3]{Attouch}. Moreover $\lim_{\lambda\to 0}f_{\lambda}(x)=f(x)$ for all $x\in X$. For a given parameter $\lambda>0$ the \textit{proximal mapping} of $f$ is defined as 
\begin{equation}
 \label{eq:proximal mapping}
 J_{\lambda}x:=\arg\min_{y\in X}\Big\{f(y)+\frac{1}{2\lambda}\|x-y\|^2\Big\}.
\end{equation}
For a function $f:X\to(-\infty,+\infty]$ let $\partial f(x)$ denote the subdifferential of $f$ at $x\in X$ 
\begin{equation}
 \label{eq:subdifferential}
 \partial f(x):=\{u\in X^*| f(x)\geqslant f(y)+\langle u,y-x\rangle,\forall y\in X\}.
\end{equation}
We say a pair $(x,u)\in X\times X^*$ lies in $\partial f$ whenever $u\in \partial f(x)$. For more on fundamental concepts in convex analysis in linear spaces refer to the classical book by Rockafellar \cite{Rockafellar}. A celebrated result in the theory of Mosco convergence is the following theorem of Attouch. 
\begin{thm}\cite[Attouch's Theorem]{Attouch}
\label{Attouch}
Let $X$ be a smooth reflexive Banach space. Let $(f^n)_{n\in\mathbb{N}},f$ be a sequence of proper closed convex functions from $X$ into $(-\infty,+\infty]$. The following equivalences hold:
 \begin{enumerate}[(i)]
  \item $M-\lim_nf^n=f$ 
  \item $\forall \lambda>0,\forall x\in X$ it holds $\lim_nJ^n_{\lambda}x=J_{\lambda}x$ and $\exists (u,v)\in\partial f, \exists (u_n,v_n)\in\partial f^n$ such that $\lim_nu_n=u$ in $X$, $\lim_nv_n=v$ in $X^*$, and $\lim_nf^n(u_n)=f(u)$
  \item $\forall\lambda>0,\forall x\in X$ it holds $\lim_nf^n_{\lambda}(x)=f_{\lambda}(x)$.
 \end{enumerate}
\end{thm}
Theorem \ref{Attouch} appeared first in \cite{Attouch1} for Hilbert spaces and then generalized for any smooth reflexive Banach space in \cite{Attouch}. 

\subsection{A theorem of Ba\v cak}
Because a norm $\|\cdot\|$ in a linear space $X$ induces a metric $d(x,y)=\|x-y\|$ for any $x,y\in X$ then definitions \eqref{eq:MYosida} and \eqref{eq:proximal mapping} can be accommodated easily in the setting of a Hadamard space using its metric. 
For a given closed convex function $f:H\to(-\infty,+\infty]$ and parameter $\lambda>0$ the {\em Moreau approximate} $f_{\lambda}$ of $f$ is defined as 
\begin{equation}
\label{eq:MY}
f_{\lambda}(x):=\inf_{y\in H}\Big\{f(y)+\frac{1}{2\lambda}d(y,x)^2\Big\},\hspace{0.3cm}\text{for each}\hspace{0.3cm} x\in H
\end{equation}
and the proximal mapping of $f$ 
\begin{equation}
\label{eq:R}
J_{\lambda}x:=\arg\min _{y\in H}\Big\{f(y)+\frac{1}{2\lambda}d(y,x)^2\Big\},\hspace{0.3cm}\text{for each}\hspace{0.3cm} x\in H.
\end{equation} 
In his study of the gradient flow in Hadamard spaces \cite{Bacak} Ba\v cak established a result which relates Mosco convergence of a sequence of closed convex functions $(f^n)_{n\in\mathbb{N}}$ to the pointwise convergence of Moreau approximates $(f^n_{\lambda})_{n\in\mathbb{N}}$ and proximal mappings $(J^n_{\lambda})_{n\in\mathbb{N}}$. 
\begin{thm}(Ba\v cak)
\label{Bacak}
Let $(H,d)$ be a Hadamard space and $f^n:H\to(-\infty,+\infty]$ a sequence of closed convex functions. If $M-\lim_nf^n(x)=f(x)$, then $\lim_n f^n_{\lambda}(x)=f_{\lambda}(x)$ and $\lim_n J^n_{\lambda}x=J_{\lambda}x$ for each $x\in H$.
\end{thm} 
This result is the analogue of the implication $(i)\to(iii)$ in Theorem \ref{Attouch}. Later Ba\v cak et al. \cite{Bacak2} proved the following. 
\begin{thm}
\label{Bacak2}
Let $(H,d)$ be a Hadamard space and $f, f^n:H\to(-\infty,+\infty]$ be a sequence of closed convex functions. If $\lim_n f^n_{\lambda}(x)=f_{\lambda}(x)$ then $M-\lim_nf^n(x)=f(x)$ for all $x\in H$.
\end{thm}
This result together with Theorem \ref{Bacak2} imply the equivalence between Mosco convergence and pointwise convergence of  Moreau approximates in Hadamard spaces. This completes the equivalence $(i)\leftrightarrow(iii)$ in Theorem \ref{Attouch} for Hadamard spaces.
However it is not known whether convergence of proximal mappings imply, under some additional conditions, the Mosco convergence of $f^n$. This was left an open question by Ba\v cak \cite{Bacak}. That convergence of proximal mappings only is not enough was noted by Ba\v cak in \cite{Bacak3}. Indeed consider a sequence of constant functions $0,1,0,1,...$ defined on $\mathbb{R}$. Evidently they are closed and convex but they don't converge in the sense of Mosco to any function $f$. However their proximal mapping maps $J_{\lambda}:\mathbb{R}\to \mathbb{R}$ (i.e. $x\mapsto J_{\lambda}x$) equal the identity map for all $\lambda>0$.  In this note we aim to complete the cycle of equivalences, the analogues of Attouch's theorem. This also answers an open question in \cite{Bacak3}.

\section{Asymptotic Boundedness for the Slope of a Sequence of Functions}

\subsection{Some preliminaries}
For a given function $f$ let $\dom f$ denote its effective domain i.e. $\dom f:=\{x\in H\quad|\quad f(x)<+\infty\}$. An element $x\in H$ is said to be a minimizer of $f$ whenever $f(x)\leqslant f(y)$ for all $y\in\dom f$.
\begin{mydef}
Let $f:H\to (-\infty,+\infty]$ be a closed convex function and $x\in\dom f$. The slope of $f$ at $x$ is defined as 
\begin{equation}
\label{eq:slope}
|\partial f|(x):=\limsup_{y\to x}\frac{\max\{f(x)-f(y),0\}}{d(x,y)}
\end{equation}
If $f(x)=+\infty$ we set $|\partial f|(x):=+\infty$.
\end{mydef}
It follows that $|\partial f|(x)= 0$ whenever $x\in H$ is a minimizer of $f$. The inclusion $\dom|\partial f|\subseteq \dom f$ is evident. Moreover the followings are true 
\begin{enumerate}[(i)]
 \item $|\partial (f+g)|(x)\leqslant |\partial f|(x)|+|\partial g|(x)$  for any two functions $f,g$
 \item $|\partial (\alpha f)|(x)=\alpha|\partial f|(x)$ for any scalar $\alpha>0$.
\end{enumerate}

\begin{lemma}\cite[Lemma 5.1.2]{Bacak}
\label{lem1}
Let $f:H\to(-\infty,+\infty]$ be a closed convex function. Then 
\begin{equation}
\label{eq:lemma}
|\partial f|(x)=\sup_{y\in H\setminus\{x\}}\frac{\max\{f(x)-f(y),0\}}{d(x,y)}, \hspace{0.2cm}x\in\dom f.
\end{equation}
Moreover $\dom|\partial f|$ is dense in $\dom f$ and $|\partial f|$ is closed whenever $f$ is closed.
\end{lemma}

 \begin{lemma}\cite[Lemma 5.1.3]{Bacak}\label{lem2}
 Let $f:H\to(-\infty,+\infty]$ be a closed convex function. Then for every $x\in H$ and $\lambda>0$ we have $J_{\lambda}x\in\dom|\partial f|$ and 
 \begin{equation}
 \label{eq:ubound}
 |\partial f|(J_{\lambda}x)\leqslant \frac{d(J_{\lambda}x,x)}{\lambda}.
 \end{equation}
 \end{lemma}
A function $f:H\to(-\infty,+\infty]$ is said to be a strongly convex function with parameter $\mu>0$ if 
\begin{equation}
 \label{strcvxfnc}
 f(x_t)\leqslant (1-t)f(x_0)+tf(x_1)-\frac{\mu}{2}d(x_0,x_1)^2
\end{equation}
for all $t\in[0,1]$. Here $x_t:=(1-t)x_0\oplus t x_1$ denotes the convex combination of $x_0$ and $x_1$. Geometrically the element $x_t$ is the unique point on the geodesic segment $[x_0,x_1]$ connecting $x_0$ with $x_1$ such that $d(x_t,x_0)=td(x_0,x_1)$ and $d(x_t,x_1)=(1-t)d(x_0,x_1)$.
\begin{proposition}\cite[Proposition 2.2.17]{Bacak}
\label{prop1}
Let $(H, d)$ be a Hadamard space and let $f: H\to (-\infty,+\infty]$ be a closed
strongly convex function with parameter $\mu>0$. Then $f$ has a unique minimizer $x\in H$ and each minimizing sequence converges to $x$. Moreover 
\begin{equation}
 \label{eq:strcvxidentity}
 f(x)+\frac{\mu}{2} d(x,y)^2\leqslant f(y),\hspace{0.2cm}\forall y\in H.
\end{equation}
\end{proposition}

\begin{proof}
 Let $(x_n)_{n\in\mathbb{N}}$ be a minimizing sequence of $f$ i.e. $\lim_nf(x_n)=\inf_{y\in H}f(y)$. By virtue of \cite[Lemma 2.2.14]{Bacak} $f$ is bounded from below. Denote by $x_{mn}:=\frac{1}{2}x_m\oplus\frac{1}{2}x_n$. By strong convexity 
 $$f(x_{mn})\leqslant \frac{1}{2}f(x_m)+\frac{1}{2}f(x_n)-\frac{\mu}{8}d(x_m,x_n)^2$$
 implying
 $$\frac{\mu}{8}d(x_m,x_n)^2\leqslant \frac{1}{2}f(x_m)+\frac{1}{2}f(x_n)-f(x_{mn}).$$
 But the new sequence $(x_{mn})_{m,n\in\mathbb{N}}$ is also a minimizing sequence. Then $\lim_{m,n}d(x_m,x_n)=0$ implies $(x_n)_{n\in\mathbb{N}}$ is Cauchy sequence so it converges to some point $x\in H$. Assumption $f$ is closed is equivalent to $f$ being lower-semicontinuous. The inequalities $f(x)\leqslant \liminf_{n}f(x_n)=\inf_{y\in H}f(y)$ and $f(x)\geqslant \inf_{y\in H}f(y)$ imply that $x\in\arg\min_{y\in H}f(y)$. Uniqueness of minimizer follows immediately from the strong convexity property. Now consider some $y\in H$ and let $\gamma:[0,1]\to H$ be the geodesic emanating from $x$ and ending at $y$ i.e. $\gamma(0)=x,\gamma(1)=y$. Then $f(x)<f(\gamma(t))$ together with the strong convexity imply 
 $$f(x)<(1-t)f(x)+tf(y)-\frac{\mu}{2}(1-t)td(x,y)^2$$
 or equivalently
 $$tf(x)<tf(y)-\frac{\mu}{2}(1-t)td(x,y)^2.$$
 Dividing by $t$ and taking limit $t\downarrow 0$ yields inequality \eqref{eq:strcvxidentity}.
\end{proof}

\subsection{Asymptotically bounded slope}
\begin{mydef}
A sequence of functions $f^n:H\to(-\infty,+\infty]$ is said to have pointwise asymptotically bounded slope on $H$ whenever $\limsup_n|\partial f^n|(x)$ is finite for all $x\in H$. If additionally for all $x\in H$ we have $\limsup_n|\partial f^n|(x)\leqslant C$ for some $C>0$ then the sequence of functions $f^n$ is said to have uniform asymptotically bounded slope on $H$.
\end{mydef}
Recall that a set $K$ of a vector space $V$ is a cone (or sometimes called a linear cone) if for each $x$ in $K$ and positive scalars $\alpha$, the product $\alpha x$ is in $K$. The set $K$ is a convex cone if and only if any nonnegative combination of elements from $K$ remains in $K$. Let $F(H)$ denote the vector space of sequences of (extended) real valued functions defined on $H$ and let $A(H):=\{(f_n)_{n\in\mathbb{N}}\in F(H)\hspace{0.1cm}|\hspace{0.1cm} \limsup_n|\partial f^n|(x)<+\infty,\forall x\in H\}$ denote the set of all sequences that have pointwise asymptotically bounded slope on $H$.
\begin{proposition}
$A(H)$ is a convex cone in $F(H)$.
\end{proposition}
\begin{proof}
 It suffices to prove the statement for only two elements. Let $(f^n),(g^n)\in A(H)$ and $\alpha,\beta>0$. Denote by $h^n:=\alpha f^n+\beta g^n$ for each $n\in\mathbb{N}$. By definition of the slope \eqref{eq:slope} we have 
 $$|\partial h^n|(x)=\limsup_{y\to x}\frac{\max\{h^n(x)-h^n(y),0\}}{d(x,y)}.$$
 On the other hand
 $$\max\{h^n(x)-h^n(y),0\}\leqslant \alpha\max\{f^n(x)-f^n(y),0\}+\beta\max\{g^n(x)-g^n(y),0\}$$
 and the fact that the limit superior of the sum is not greater than the sum of limit superior together with $\alpha,\beta>0$ imply
 $$|\partial h^n|(x)\leqslant \alpha\limsup_{y\to x}\frac{\max\{f^n(x)-f^n(y),0\}}{d(x,y)}+\beta\limsup_{y\to x}\frac{\max\{g^n(x)-g^n(y),0\}}{d(x,y)}$$
 or equivalently 
 $$|\partial h^n|(x)\leqslant\alpha |\partial f^n|(x)+\beta |\partial g^n|(x),\hspace{0.2cm}\forall n\in\mathbb{N}.$$
 Taking limit superior with respect to $n$ on both sides yields 
 $$\limsup_n|\partial h^n|(x)\leqslant\limsup_n(\alpha |\partial f^n|(x)+\beta |\partial g^n|(x))\leqslant \alpha \limsup_n|\partial f^n|(x)+\beta \limsup_n|\partial g^n|(x).$$
 Assumption $(f^n),(g^n)\in A(H)$ implies $\limsup_n|\partial f^n|(x),\limsup_n|\partial g^n|(x)<+\infty,\forall x\in H$. Hence $\limsup_n|\partial h^n|(x)|<+\infty$ for each $x\in H$ gives $(h^n)\in A(H)$ as desired.
\end{proof}
\begin{remark}
 The set $A_0(H)$ of sequences of functions with uniform asymptotically bounded slope is also a convex cone.
\end{remark}

\begin{proposition} Let $(f^n)$ be a sequence of proper closed convex functions defined on a Hadamard space $(H,d)$. Let $f$ be the pointwise limit of $(f^n)$ such that $\dom|\partial f|\neq\emptyset$. For a given element $x\in H$ define the sequence of functions $(g^n)$ and $g$ for all $y\in H\setminus\{x\}$
\begin{align*}
\label{eq:ratio}
& g^n(y;x):=\frac{\max\{f^n(x)-f^n(y),0\}}{d(x,y)},\hspace{0.2cm}n\in\mathbb{N}\\\nonumber
& g(y;x):=\frac{\max\{f(x)-f(y),0\}}{d(x,y)}.
\end{align*} Then $(f^n)\in A(\dom|\partial f|)$ whenever 
\begin{equation}
\label{eq:sufficientcondition}
\lim_n\sup_{y\in H\setminus\{x\}}|g^n(y;x)-g(y;x)|=0.
\end{equation}
If additionally $\sup_{x\in \dom|\partial f|}|\partial f|(x)<+\infty$ then $(f^n)\in A_0(\dom|\partial f|)$.
\end{proposition}
\begin{proof}
From the elementary reverse triangle inequality 
$$\sup_{y\in H\setminus\{x\}}|g^n(y;x)-g(y;x)|\geqslant |\sup_{y\in H\setminus\{x\}}g^n(y;x)-\sup_{y\in H\setminus\{x\}}g(y;x)|.$$
Assumption \eqref{eq:sufficientcondition} implies $\lim_n\sup_{y\in H\setminus\{x\}}g^n(y;x)=\sup_{y\in H\setminus\{x\}}g(y;x)$. By virtue of Lemma \ref{lem1} this is equivalent to $\lim_n|\partial f^n|(x)=|\partial f|(x)$. Since $\dom |\partial f|\neq\emptyset$ then $\lim_n|\partial f^n|(x)$ is finite on $\dom|\partial f|$. Therefore $(f^n)$ has pointwise asymptotically bounded slope on $\dom|\partial f|$. If additionally $\sup_{x\in \dom|\partial f|}|\partial f|(x)<+\infty$ then $|\partial f|(x)\leqslant C$ for some $C>0$ for all $x\in\dom|\partial f|$. This implies $\lim_n|\partial f^n|(x)\leqslant C$ for all $x\in \dom|\partial f|$. 
\end{proof}

\section{A Converse Theorem}

\begin{thm}
\label{thm1}
Let $(H,d)$ be a Hadamard space and $f^n:H\to(-\infty,+\infty]$ be a sequence of closed convex functions. Suppose 
\begin{enumerate}[(i)]
\item $\lim_nf^n(x)=f(x)$ for all $x\in H$\label{a1}
\item $(f^n)\in A(H)$ \label{a2}
\end{enumerate}
If $\lim_n J^n_{\lambda}x=J_{\lambda}x$ then $\lim_nf^n_{\lambda}(x)=f_{\lambda}(x)$ for each $x\in H$.
\end{thm}

\begin{proof} Note that $f^n$ is convex for each $n$. Since the metric $d(\cdot,x)^2$ is a strongly convex function then the map
$$y\mapsto f^n(y)+\frac{1}{2\lambda}d(y,x)^2$$
is strongly convex for each $x\in H$. It follows from Proposition \ref{prop1} that the proximal mapping 
$$J^n_{\lambda}x:=\arg\min _{y\in H}\Big\{f^n(y)+\frac{1}{2\lambda}d(y,x)^2\Big\}$$ exists and it is unique. Similarly for $J_{\lambda}x$. By definition for all $n$ we have
$$f^n_{\lambda}(x)=f^n(J^n_{\lambda}x)+\frac{1}{2\lambda}d(x,J^n_{\lambda}x)^2.$$
From the elementary triangle inequality
$d(x,J^n_{\lambda}x)\leqslant d(x,J_{\lambda}x)+d(J_{\lambda}x,J^n_{\lambda}x)$
and interchanging the role of $J^n_{\lambda}x$ with $J_{\lambda}x$ we obtain the estimate
$$|d(x,J^n_{\lambda}x)-d(x,J_{\lambda}x)|\leqslant d(J_{\lambda}x,J^n_{\lambda}x).$$
Assumption $\lim_nJ^n_{\lambda}x=J_{\lambda}x$ implies $\lim_nd(x,J^n_{\lambda}x)=d(x,J_{\lambda}x)$ for each $x\in H$. Therefore it is sufficient to prove $\lim_nf^n(J^n_{\lambda}x)=f(J_{\lambda}x)$. By Lemma \ref{lem2}, $J_{\lambda}x\in \dom |\partial f|$ for any $x\in H$  yields $J_{\lambda}x\in \dom f$ since $\dom |\partial f|\subseteq\dom f$. Similarly $J^n_{\lambda}x\in \dom f^n$.  
From the definition of Moreau approximate it follows that for all $n$
$$f^n(J^n_{\lambda}x)+\frac{1}{2\lambda}d(J^n_{\lambda}x,x)^2\leqslant f^n(J_{\lambda}x)+\frac{1}{2\lambda}d(J_{\lambda}x,x)^2.$$
which in turn together with assumption \eqref{a1} and $\lim_n J^n_{\lambda}x=J_{\lambda}x$ gives 
\begin{equation}
\label{eq:limsup}
-\infty\leqslant\limsup_nf^n(J^n_{\lambda}x)\leqslant  f(J_{\lambda}x)<+\infty.
\end{equation}
On the other hand assumption \eqref{a2} implies that for some finite valued nonnegative function $C:H\to\mathbb{R}_+$ we have $\limsup_n|\partial f^n|(x)\leqslant C(x)$ for all $x\in H$. In particular $\limsup_n|\partial f^n|(J_{\lambda}x)\leqslant C(J_{\lambda}x)<+\infty$ for all $x\in H$. Therefore there exists some $n_0\in\mathbb{N}$ such that for all $n\geqslant n_0$ we have $J_{\lambda}x\in\dom |\partial f^n|$ implying that $f^n(J_{\lambda}x)$ and $|\partial f^n|(J_{\lambda}x)$ are finite. 
By virtue of Lemma \ref{lem1} the following inequality holds for all $n\geqslant n_0$ 
$$f^n(J^n_{\lambda}x)\geqslant f^n(J_{\lambda}x)-|\partial f^n|(J_{\lambda}x) d(J_{\lambda}x,J^n_{\lambda}x).$$
This implies 
\begin{align}
\label{eq:lesh}
+\infty>\liminf_nf^n(J^n_{\lambda}x)&\geqslant f(J_{\lambda}x)-\limsup_n|\partial f^n|(J_{\lambda}x) d(J_{\lambda}x,J^n_{\lambda}x)\geqslant-\infty.
\end{align}
But $\limsup_n|\partial f^n|(J_{\lambda}x)\leqslant C(J_{\lambda}x)<+\infty$ yields 
$$\limsup_n|\partial f^n|(J_{\lambda}x) d(J_{\lambda}x,J^n_{\lambda}x)=\limsup_n|\partial f^n|(J_{\lambda}x)\cdot \lim_n d(J_{\lambda}x,J^n_{\lambda}x)\leqslant C(J_{\lambda}x)\cdot 0=0$$
which together with \eqref{eq:lesh} gives 
\begin{equation}
\label{eq:liminf}
+\infty>\liminf_nf^n(J^n_{\lambda}x)\geqslant f(J_{\lambda}x)>-\infty.
\end{equation}
From inequality \eqref{eq:liminf} and \eqref{eq:limsup} we obtain $f(J_{\lambda}x)=\lim_nf^n(J^n_{\lambda}x)$ as required.
\end{proof}
It is natural to ask if, under some additional condition, the pointwise convergence of $f^n$ to $f$ is also a necessary condition. The following theorem establishes this.

\begin{thm}
\label{thm2}
Let $(H,d)$ be a Hadamard space and $f,f^n:H\to(-\infty,+\infty]$ be a sequence of closed convex functions on $H$. Suppose $(f^n)\in A(H)$. If for all $x\in H$, $\lim_nf^n_{\lambda}(x)=f_{\lambda}(x)$  then 
\begin{enumerate}[(i)]
\item $\lim_n J^n_{\lambda}x=J_{\lambda}x$\label{r1}
\item $\lim_nf^n(x)=f(x)$.\label{r2}
\end{enumerate}
\end{thm}
\begin{proof}
By Theorem \ref{Bacak2} assumption $\lim_nf^n_{\lambda}(x)=f_{\lambda}(x)$ implies $M-\lim_nf^n(x)=f(x)$ for all $x\in H$. Then Theorem \ref{Bacak} yields $\lim_nJ^n_{\lambda}x=J_{\lambda}x$ for all $x\in H$. This proves \eqref{r1}
which in turn yields
$$f_{\lambda}(x)=\lim_nf^n_{\lambda}(x)=\limsup_nf^n(J^n_{\lambda}x)+\frac{1}{2\lambda}\lim_nd(J^n_{\lambda}x,x)^2=\limsup_nf^n(J^n_{\lambda}x)+\frac{1}{2\lambda}d(J_{\lambda}x,x)^2.$$
By definition of Moreau approximate then it follows $f(J_{\lambda}x)=\limsup_nf^n(J^n_{\lambda}x)$. Similarly 
$f(J_{\lambda}x)=\liminf_nf^n(J^n_{\lambda}x)$
hence $f(J_{\lambda}x)=\lim_nf^n(J^n_{\lambda}x)$. On the other hand for each $n\in\mathbb{N}$ we have
$$f^n(J^n_{\lambda}x)\leqslant f^n(J^n_{\lambda}x)+\frac{1}{2\lambda}d(J^n_{\lambda}x,x)^2\leqslant f^n(x)\Rightarrow\lim_n f^n(J^n_{\lambda}x)\leqslant\liminf_n f^n(x).$$
Therefore $f(J_{\lambda}x)\leqslant\liminf_n f^n(x)$ for all $x\in H$ and for all $\lambda>0$. Using $\lim_{\lambda\downarrow 0}J_{\lambda}x=x$ and the assumption that $f$ is closed we obtain 
\begin{equation}
\label{eq:lower}
f(x)\leqslant \liminf_{\lambda\downarrow 0}f(J_{\lambda}x)\leqslant\liminf_n f^n(x).
\end{equation}
By \cite[Lemma 1.18]{Attouch} there exists a mapping $n\mapsto \lambda(n)$ such that $\lim_n\lambda(n)=0$ and $$\lim_{\lambda\downarrow 0}\lim_nf^n_{\lambda}(x)=\lim_nf^n_{\lambda(n)}(x).$$
By definition of Moreau approximate we can write
 $$f^n_{\lambda(n)}(x)=f^n(J^n_{\lambda(n)}x)+\frac{1}{2\lambda(n)}d(J^n_{\lambda(n)}x,x)^2$$
 implying 
\begin{equation}
 \label{eq:lineq}
 f(x)\geqslant \lim_n\Big[f^n(J^n_{\lambda(n)}x)+\frac{1}{2\lambda(n)}d(J^n_{\lambda(n)}x,x)^2\Big]\geqslant \limsup_nf^n(J^n_{\lambda(n)}x).
 \end{equation}
By Lemma \ref{lem1} we have the inequalities
\begin{equation}
\label{eq:inequality}
f^n(J^n_{\lambda(n)}x)+|\partial f^n|(x)d(J^n_{\lambda(n)}x, x)\geqslant f^n(x),\hspace{0.2cm}\forall n\in\mathbb{N}
\end{equation}
which then give 
\begin{equation}
\label{eq:inequality1}
\limsup_nf^n(J^n_{\lambda(n)}x)+\limsup_n|\partial f^n|(x)d(J^n_{\lambda(n)}x, x)\geqslant \limsup_nf^n(x).
\end{equation}
Assumption  $(f^n)$ has pointwise asymptotically bounded slope on $H$ implies that for some nonnegative finite valued function $C:H\to\mathbb{R}_+$ we have $\limsup_n|\partial f^n|(x)\leqslant C(x)$. Hence 
$$0\leqslant\limsup_n|\partial f^n|(x)d(J^n_{\lambda(n)}x, x)=\limsup_n|\partial f^n|(x)\cdot \lim_n d(J^n_{\lambda(n)}x, x)\leqslant C(x)\cdot 0=0.$$
From inequalities \eqref{eq:lineq} and \eqref{eq:inequality1} it follows
\begin{equation}
\label{eq:upper}
f(x)\geqslant \limsup_nf^n(J^n_{\lambda(n)}x)\geqslant \limsup_nf^n(x).
\end{equation}
The inequalities \eqref{eq:lower} and \eqref{eq:upper} imply $f(x)=\lim_nf^n(x)$.
\end{proof}
It was pointed out by Ba\v cak that Thoerem \ref{thm2} \eqref{r2} can be proved directly by employing the following two key lemmas.
\begin{lemma}\cite[Proposition 2.2.26]{Bacak}
 \label{keylemma1}
 Let $f:H\to(-\infty,+\infty]$ be a closed convex function and $x\in H$. Then the function $\lambda\mapsto J_{\lambda}x$ is continuous on $(0,+\infty)$ and 
 \begin{equation}
  \label{id1}
  \lim_{\lambda\downarrow 0}J_{\lambda}x=P_{\cl\dom f}x.
 \end{equation}
In particular if $x\in \cl\dom f$ then $\lambda\mapsto J_{\lambda}x$ is continuous on $[0,+\infty)$.
\end{lemma}

\begin{lemma}\cite[Lemma 5.1.4]{Bacak}
 \label{keylemma2}
 Let $f:H\to (-\infty,+\infty]$ be a closed convex function. Then for any $x\in H$ and $\lambda\in(0,+\infty)$ we have 
 \begin{equation}
  \label{id2}
  \frac{f(x)-f_{\lambda}(x)}{\lambda}\leqslant\frac{|\partial f|^2(x)}{2}.
 \end{equation}
\end{lemma}
Without loss of generality let $x\in\cl\dom f$. From triangle inequality for each $n\in\mathbb{N}$ we have the upper estimate
\begin{equation}
 \label{triangleineq}
 |f^n(x)-f(x)|\leqslant|f^n(x)-f^n_{\lambda}(x)|+|f^n_{\lambda}(x)-f_{\lambda}(x)|+|f_{\lambda}(x)-f(x)|.
\end{equation}
By Lemma \ref{keylemma2} we have $|f^n(x)-f^n_{\lambda}(x)|\leqslant \lambda |\partial f^n|^2(x)/2$ and for sufficiently large $n$ assumption $(f^n)\in A(H)$ implies $|f^n(x)-f^n_{\lambda}(x)|\leqslant \lambda C(x)$ for some finite valued function $C(x)$. Hence this term vanishes as $\lambda\downarrow 0$. 
The middle term in \eqref{triangleineq} vanishes by assumption $\lim_nf^n_{\lambda}(x)=f_{\lambda}(x)$ for each $x\in H$.
On the other hand Lemma \ref{keylemma1} implies $\lim_{\lambda\downarrow 0}J_{\lambda}x=x$. 
The evident chain of inequalities $f(J_{\lambda}x)\leqslant f_{\lambda}(x)\leqslant f(x)$ together with lsc of $f$ imply $|f_{\lambda}(x)-f(x)|\to 0$ as $\lambda\downarrow 0$.
An application of Theorem \ref{Bacak} and Theorem \ref{Bacak2} yield the following. 
\begin{thm}
\label{mainthm}
Let $(H,d)$ be a Hadamard space and $f, f^n:H\to(-\infty,+\infty]$ be a sequence of proper closed convex functions. If $(f^n)\in A(H)$, then $M-\lim_nf^n=f$ if and only if $\lim_n f^n(x)=f(x)$ and $\lim_nJ^n_{\lambda}x=J_{\lambda}x$ for each $x\in H$.
\end{thm}
\begin{proof}
 Assume $(f^n)\in A(H)$ and let $\lim_nf^n(x)=f(x)$ for all $x\in H$. Then by Theorem \ref{thm1} $\lim_nJ^n_{\lambda}x=J_{\lambda}x$ implies $\lim_nf^n_{\lambda}(x)=f_{\lambda}(x)$ for all $x\in H$. Theorem \ref{Bacak2} in turn yields $M-\lim_nf^n(x)=f(x)$. Now suppose $M-\lim_nf^n(x)=f(x)$ then by Theorem \ref{Bacak} we get $\lim_nf^n_{\lambda}(x)=f_{\lambda}(x)$ for each $x\in H$. Since by assumption $(f^n)\in A(H)$ then Theorem \ref{thm2} implies $\lim_nf^n(x)=f(x)$ and $\lim_nJ^n_{\lambda}x=J_{\lambda}x$ for all $x\in H$.
\end{proof}

\subsection{A normalization condition}
Let $f^n,f:H\to(-\infty,+\infty]$ be a family of proper closed convex functions. We say the sequence of functions $(f^n)_{n\in\mathbb{N}}$ satisfies the {\em normalization condition} if there exists some sequence $(x_n)_{n\in\mathbb{N}}\subset H$ and $x\in H$ such that $x_n\to x, f^n(x_n)\to f(x)$ and $|\partial f^n|(x_n)\to|\partial f|(x)$ as $n\uparrow+\infty$. For a sequence of functions $(f^n)_{n\in\mathbb{N}}$ that Mosco converges to some function $f$ we get the following result.
\begin{lemma}
 \label{normalization1}
 A sequence of closed convex functions $(f^n)_{n\in\mathbb{N}},f:H\to(-\infty,+\infty]$ satisfies the normalization condition whenever $M-\lim_nf^n=f$. 
\end{lemma}
\begin{proof}
 Let $x_0\in H$ then $M-\lim_nf^n=f$ implies by Theorem \ref{Bacak} we have $\lim_nJ^n_{\lambda}x_0=J_{\lambda}x_0$ for any $\lambda>0$. Take $x_n:=J^n_{\lambda}x_0$ and $x:=J_{\lambda}x_0$. Then this means $\lim_nx_n=x$. We need to show the other two properties. Note that by definition of the proximal mapping $J_{\lambda}$ we have 
 $$f^n(x_n)+\frac{1}{2\lambda}d(x_0,x_n)^2\leqslant f^n(y)+\frac{1}{2\lambda}d(x_0,y)^2,\quad\forall y\in H.$$
 Let $(\xi_n)_{n\in\mathbb{N}}\subset H$ be a sequence strongly converging to $x$. From the last inequality we obtain in particular that 
 $$f^n(x_n)+\frac{1}{2\lambda}d(x_0,x_n)^2\leqslant f^n(\xi_n)+\frac{1}{2\lambda}d(x_0,\xi_n)^2,\quad\forall n\in \mathbb{N}$$
 implying $\limsup_nf^n(x_n)\leqslant\limsup_nf^n(\xi_n)_{n\in\mathbb{N}}$. On the other hand by definition of Mosco convergence we can have $(\xi_n)_{n\in\mathbb{N}}$ such that $\limsup_nf^n(\xi_n)\leqslant f(x)$. Hence $\limsup_nf^n(x_n)\leqslant f(x)$. Moreover $\lim_nx_n=x$ implies in particular that $x_n\overset{w}\to x$. Again by definition of Mosco convergence we obtain $f(x)\leqslant\liminf_nf^n(x_n)$. Therefore $f(x)=\lim_nf^n(x_n)$ as desired. Next we need to show the property about the slopes. Note that by Lemma \ref{lem1} we have 
 $$\frac{\max\{f^n(x_n)-f^n(y),0\}}{d(x_n,y)}\leqslant|\partial f^n|(x_n),\quad\forall y\in H,\forall n\in\mathbb{N}.$$
 Again by Mosco convergence for each $y\in H$ there is a sequence $(\xi_n)_{n\in\mathbb{N}}$ strongly converging to $y$ such that $\limsup_nf^n(\xi_n)\leqslant f(y)$. Applying the last inequality for $\xi_n$ we have 
 $$\frac{\max\{f^n(x_n)-f^n(\xi_n),0\}}{d(x_n,y)}\leqslant|\partial f^n|(x_n),\quad\forall n\in\mathbb{N}$$
 which in turn yields $$\frac{\max\{f(x)-\limsup_nf^n(\xi_n),0\}}{d(x,y)}\leqslant\liminf_n|\partial f^n|(x_n).$$
 Using $\limsup_nf^n(\xi_n)\leqslant f(y)$ we get
 $$\frac{\max\{f(x)-f(y),0\}}{d(x,y)}\leqslant\liminf_n|\partial f^n|(x_n).$$
 Because the last inequality holds for any $y\in H$ then $|\partial f|(x)\leqslant\liminf_n|\partial f^n|(x_n)$. Now by definition \eqref{eq:slope} we obtain 
 $$|\partial f^n|(x_n)\leqslant \frac{\max\{f^n(x_n)-f^n(y_n),0\}}{d(x_n,y_n)}+\varepsilon_n,\quad\forall n\in\mathbb{N}$$
 for sufficiently small $\varepsilon_n>0$ and $y_n$ sufficiently close to $x_n$. Note that strong convergence of $x_n$ to $x$ implies that for any $\delta>0$ all but finitely many of the terms $y_n\in \mathbb{B}(x,\delta)$. In particular $(y_n)$ is a bounded sequence hence it has a weakly convergent subsequence $(y_{n_k})$. But $\cl\mathbb{B}(x,\delta)$ is a closed convex set and  since weak convergence coincides on bounded sets with the so called $\Delta$-convergence (see \cite[Chapter 3]{Berdellima}) then by \cite[Lemma 3.2.1]{Bacak} $y_{n_k}\overset{w}\to y\in\cl\mathbb{B}(x,\delta)$. 
 One can choose $(\varepsilon_n)$ such that $\lim_{k}\varepsilon_{n_k}=0$.  Moreover $d(x,\cdot)$ is weakly lsc (\cite[Corollary 3.2.4]{Bacak} implying
 $$\limsup_k|\partial f^{n_k}|(x_{n_k})\leqslant\frac{\max\{f(x)-\liminf_kf^{n_k}(y_{n_k}),0\}}{d(x,y)}.$$
 By definition of Mosco convergence follows $\liminf_nf^n(y_n)\geqslant f(y)$. Hence 
 $$\limsup_n|\partial f^n|(x_n)\leqslant\limsup_k|\partial f^{n_k}|(x_{n_k})\leqslant\frac{\max\{f(x)-f(y),0\}}{d(x,y)}.$$
 The last inequality implies $\limsup_n|\partial f^n|(x_n)\leqslant|\partial f|(x)$. 
\end{proof}
A family of functions $f^n:H\to(-\infty,+\infty]$ is said to be {\em equi locally Lipschitz} if for any bounded set $K\subseteq H$ there is a constant $C_K>0$ such that 
\begin{equation}
 \label{equicontinuity}
 |f^n(x)-f^n(y)|\leqslant C_Kd(x,y),\quad\forall x,y\in K, \forall n\in\mathbb{N}.
\end{equation}

\begin{lemma}
 \label{l:equicontinuity}
 Let $f^n:H\to(-\infty,+\infty]$ be a sequence of closed convex functions such that $\lim_nf^n_{\lambda}(x_0)=\alpha_0\in\mathbb{R}$ for some $x_0\in H$ and some $\lambda>0$. Then $(f^n_{\lambda})_{n\in\mathbb{N}}$ are equi locally Lipschitz functions. 
\end{lemma}
\begin{proof}
 By virtue of \cite[Theorem 2.64 $(ii)$]{Attouch} it suffices to show that there is $r>0$ and $x_0\in H$ such that $f^n(x)+r(d(x,x_0)^2+1)\geqslant 0$ for all $x\in H$ and all $n\in\mathbb{N}$. Let $x_0\in H$ be such that $\lim_nf^n_{\lambda}(x_0)=\alpha_0\in\mathbb{R}$. Notice that by definition of Moreau envelope we have
 $$f^n(x)\geqslant f^n_{\lambda}(x_0)-\frac{1}{2\lambda}d(x_0,x)^2\geqslant \alpha_0-\delta-\frac{1}{2\lambda}d(x_0,x)^2$$
 for some $\delta>0$ and sufficiently large $n$. If one takes $\delta=\alpha_0+1/2\lambda$ then one gets
 $$f^n(x)\geqslant -\frac{1}{2\lambda}(d(x_0,x)^2+1),\quad\forall x\in H.$$
 For any $r\geqslant 1/2\lambda$ we obtain $f^n(x)+r(d(x_0,x)^2+1)\geqslant0$ for all $x\in H$ and all $n\in\mathbb{N}$.
\end{proof}

Let $f:H\to(-\infty,+\infty]$. The {\em geodesic lower directional derivative} of $f$ at $x\in H$ along a geodesic $\gamma\in\Gamma_x(H)$ is defined as 
\begin{equation}
 \label{lowerderivative}
 f'_{-}(x;\gamma):=\liminf_{y\overset{\gamma}\to x}\frac{f(y)-f(x)}{d(y,x)}.
\end{equation}
Analogously the {\em  geodesic upper directional derivative}, denoted by $f'_{+}(x;\gamma)$, is defined with liminf replaced by limsup. If both limits exist and coincide then we say $f$ is geodesically differentiable at $x$ along $\gamma\in\Gamma_x(H)$ and denote it by $f'(x;\gamma)$.

\begin{thm}[Attouch's Theorem for Hadamard spaces]
 \label{th:allaAttouch}
 Let $f^n,f:H\to(-\infty,+\infty]$ be a sequence of closed convex functions such that
 \begin{enumerate}[(i)]
  \item $\forall \lambda>0,\forall x\in H$ it holds $\lim_nJ^n_{\lambda}x=J_{\lambda}x$\label{a:ass1}
  \item $(f^n)_{n\in\mathbb{N}}$ satisfies the normalization condition\label{a:ass2} with $(x_n)_{n\in\mathbb{N}}$ such that $x_n\to x_0\in H$
  \item $\lim_nf'_{n,\lambda}(x_t;\gamma)=f'_{\lambda}(x_t;\gamma)$ for all $\gamma\in\Gamma_{x_0}(H)$ and $x_t\in\gamma$ where $t\in[0,1]$\label{a:ass3}.
 \end{enumerate}
Then $\forall \lambda>0,\forall x\in H$ it holds $\lim_nf^n_{\lambda}(x)=f_{\lambda}(x)$.
\end{thm}
\begin{proof}
 Let $(f^n)_{n\in\mathbb{N}},f$ satisfy the normalization condition. Then there exists $(x_n),x_0\subset H$ such that $\lim_nx_n=x_0,\lim_nf^n(x_n)=f(x_0)$ and $\lim_n|\partial f^n|(x_n)=|\partial f|(x_0)$.  Let $\lambda>0$. First we claim that $\lim_nf^n_{\lambda}(x_0)=f_{\lambda}(x_0)$. Introduce the variables $u_n:=J^n_{\lambda}x_n$ for each $n\in\mathbb{N}$ and $u_0:=J_{\lambda}x_0$. Note that by assumption \eqref{a:ass1} for each fixed $m\in\mathbb{N}$ we have $\lim_nJ^n_{\lambda}x_m=J_{\lambda}x_m$. Since the mapping $x\mapsto J_{\lambda}x$ is nonexpansive and therefore continuous, then $\lim_mJ_{\lambda}x_m=J_{\lambda}x_0$. By triangle inequality $d(J^n_{\lambda}x_n,J_{\lambda}x_0)\leqslant d(J^n_{\lambda}x_n,J^n_{\lambda}x_m)+d(J^n_{\lambda}x_m,J_{\lambda}x_0)$ and nonexpansiveness of $J^n_{\lambda}$ we have
 $$d(J^n_{\lambda}x_n,J_{\lambda}x_0)\leqslant d(x_n,x_m)+d(J^n_{\lambda}x_m,J_{\lambda}x_0).$$
 Passing in the limit as $m,n\uparrow+\infty$ we obtain $\lim_nu_n=\lim_nJ^n_{\lambda}x_n=J_{\lambda}x_0=u_0$.
 On the other hand  
 $$|f^n(u_n)-f(u_0)|\leqslant |f^n(u_n)-f^n(x_n)|+|f^n(x_n)-f(x_0)|+|f(x_0)-f(u_0)|.$$
 By normalization condition and using $\lim_{\lambda\downarrow 0}u_n=\lim_{\lambda\downarrow 0}J^n_{\lambda}x_n=x_n, \lim_{\lambda\downarrow 0}u_0=\lim_{\lambda\downarrow 0}J_{\lambda}x_0=x_0$ and lsc of $f^n$ and $f$ implies in the limit as $\lambda\downarrow0$ and $n\uparrow+\infty$ that $\lim_nf^n(u_n)=f(u_0)$. Again by definition of Moreau envelope 
 $$f^n_{\lambda}(x_n)=f^n(u_n)+\frac{1}{2\lambda}d(x_n,u_n)^2\to f(u_0)+\frac{1}{2\lambda}d(x_0,u_0)^2:=f_{\lambda}(x_0),\quad\text{as}\quad n\uparrow+\infty.$$
 Note that 
 $$f^n_{\lambda}(x_0)\leqslant f^n(x_n)+\frac{1}{2\lambda}d(x_0,x_n)^2\to f(x_0)\quad\text{as}\quad n\uparrow+\infty.$$
 On the other hand we have 
 \begin{align*}
  f^n_{\lambda}(x_0)\geqslant f^n(J^n_{\lambda}x_0)\geqslant f^n(x_n)&-|\partial f^n|(x_n)d(J^n_{\lambda}x_0,x_n)\\&\to f(x_0)-|\partial f|(x_0)d(J_{\lambda}x_0,x_0)>-\infty\quad\text{as}\quad n\uparrow+\infty.
 \end{align*}
 In particular we obtain that $-\infty<\liminf_nf^n_{\lambda}(x_0)\leqslant\limsup_nf^n_{\lambda}(x_0)<+\infty$ (one can assume that $x_0\in\dom f$ else there is nothing to show). By Lemma \ref{l:equicontinuity} we get that $(f^n_{\lambda})_{n\in\mathbb{N}}$ is equi locally Lipschitz in $H$. This means that for any bounded domain $K\subseteq H$ there is $C_K>0$ such that 
 $$|f^n_{\lambda}(x)-f^n_{\lambda}(y)|\leqslant C_Kd(x,y),\quad\forall x,y\in K,\forall n\in\mathbb{N}.$$
From this and the estimate
$$|f^n_{\lambda}(x_0)-f_{\lambda}(x_0)|\leqslant |f^n_{\lambda}(x_0)-f^n_{\lambda}(x_n)|+|f^n_{\lambda}(x_n)-f_{\lambda}(x_0)|\leqslant C_Kd(x_n,x_0)+|f^n_{\lambda}(x_n)-f_{\lambda}(x_0)|.$$
follows $\lim_nf^n_{\lambda}(x_0)=f_{\lambda}(x_0)$. Now define $g_{n,\lambda}(t):=f^n_{\lambda}(x_t)$ where $x_t:=(1-t)x_0\oplus tx$ and $x\in H$ is arbitrary. Consider 
$$g'_{n,\lambda}(t):=\lim_{s\to 0}\frac{g_{n,\lambda}(t+s)-g_{n,\lambda}(s)}{s}.$$
Since $f^n_{\lambda}$ is convex for each $n\in\mathbb{N}$ then it is absolutely continuous on every geodesic segment. In particular $g'_{n,\lambda}(t)$ exists almost everywhere on $[0,1]$, it is Lebesgue integrable on the interval $[0,1]$ and satifies 
\begin{equation}
 \label{Lebesgue}
 f^n_{\lambda}(x)=f^n_{\lambda}(x_0)+\int^1_0g'_{n,\lambda}(t)\,dt.
\end{equation}
On the other hand $g'_{n,\lambda}(t)=f'_{n,\lambda}(x_t;\gamma)d(x_0,x)$ where $\gamma\in\Gamma_{x_0}(H)$ connects $x_0$ with $x$ and $x_t\in\gamma$. Assumption \eqref{a:ass3} implies $\lim_ng'_{n,\lambda}(t)=g'_{\lambda}(t)$ for all $t\in[0,1]$. Moreover equi locally Lipschitz property of $(f^n_{\lambda})_{n\in\mathbb{N}}$ implies that $\sup_ng'_{n,\lambda}(t)\leqslant C_K d(x_0,x)$ for any bounded domain $K$ around $x_0$ and $x\in K$. By Lebesgue dominated convergence theorem we obtain in the limit 
$$\lim_nf^n_{\lambda}(x)=f_{\lambda}(x_0)+\int^1_0\lim_ng'_{n,\lambda}(t)\,dt=f_{\lambda}(x_0)+\int^1_0g'_{\lambda}(t)\,dt=f_{\lambda}(x).$$
 
\end{proof}

 \bibliographystyle{amsplain}
  \bibliography{literature}

% \begin{thebibliography}{99}
% 
% \bibitem{Attouch1} H. Attouch.  {\it Famille d'op\'erateurs maximaux monotones et m\'esurabilit\'e}, Ann. Mat. Pura Appl., (4) 120, 35-111 1979.
% \bibitem{Attouch} H. Attouch.  {\it Variational convergence for functions and operators}, Applicable Mathematics Series, Pitman (Advanced Publishing Program), Boston, MA, 1984.
% \bibitem{Bacak} M. Ba\v cak.  {\it Convex analysis and optimization in Hadamard spaces}, vol. 22 of De Gruyter Series in Nonlinear Analysis and Applications, De Gruyter, Berlin, 2014.
% \bibitem{Bacak2} M. Ba\v cak, M. Montag, and G. Steidl.  {\it Convergence of functions and their Moreau envelopes on Hadamard spaces}, Journal of Approximation Theory, (C) 224, 1-12 2017.
% \bibitem{Bacak3} M. Ba\v cak.  {\it Old and new challenges in Hadamard spaces}, arHiv:1807.01355 2018.
% \bibitem{BT} F. Bruhat, J. Tits.  {\it Groupes r\' eductifs sur un corps local. I. Donne\'es radicielles value\'es}, Inst. Hautes \' Etudes Sci. Publ. Math. 41, 5-251 1972.
% 
% \end{thebibliography}

\end{document}